\documentclass[]{spie}

\newcommand{\C}{\mathbb{C}}

\usepackage{amsmath,amsfonts,amssymb}
\usepackage{graphicx}
\usepackage[colorlinks=true, allcolors=blue]{hyperref}
\usepackage{epstopdf}
\usepackage[noabbrev,capitalise]{cleveref}
\usepackage{algorithmic}

\title{Biangular Gabor frames and Zauner's conjecture}

\author{Mark Magsino}
\author{Dustin G. Mixon}

\affil{Department of Mathematics, The Ohio State University, Columbus, OH 43210}

\authorinfo{Send correspondence to \texttt{magsino.2@osu.edu} and \texttt{mixon.23@osu.edu}}

\newtheorem{thm}{Theorem}

\newtheorem{lem}[thm]{Lemma}
\newtheorem{cor}[thm]{Corollary}
\newtheorem{example}[thm]{Example}
\newtheorem{problem}[thm]{Problem}

\crefname{prop}{Proposition}{Propositions}
\begin{document}
\maketitle

\begin{abstract}
Two decades ago, Zauner conjectured that for every dimension $d$, there exists an equiangular tight frame consisting of $d^2$ vectors in $\C^d$.
Most progress to date explicitly constructs the promised frame in various dimensions, and it now appears that a constructive proof of Zauner's conjecture may require progress on the Stark conjectures.
In this paper, we propose an alternative approach involving biangular Gabor frames that may eventually lead to an unconditional non-constructive proof of Zauner's conjecture.
\end{abstract}

\section{Introduction}
Let $F = \{f_j\}_{j=1}^n$ denote a finite sequence of vectors in $\C^d$.
We say $F$ is a \textbf{frame} for $\C^d$ if there exist $A,B > 0$ such that for every $x \in \C^d$, it holds that
\[
A \|x\|_2^2
\leq \sum_{j=1}^n |\langle x, f_j \rangle|^2
\leq B \|x\|_2^2.
\]
We say $F$ is \textbf{tight} if one may take $A = B$, and we say $F$ is \textbf{unit norm} if $\|f_j\|_2 = 1$ for every $j$.
Finally, we say a unit norm $F$ is \textbf{equiangular} if there exists $\alpha \geq 0$ such that $|\langle f_j, f_{j'} \rangle|^2 = \alpha$ whenever $j\neq j'$.
The lines spanned by the vectors in an equiangular tight frame (ETF) happen to form an optimal packing of points in projective space, as they achieve equality in the so-called Welch bound \cite{welch1974lower}.
As an artifact of this optimality, ETFs enjoy applications in compressed sensing \cite{bandeira2013road}, digital fingerprinting \cite{mixon2013fingerprinting}, multiple description coding \cite{strohmer2003grassmannian}, and quantum state tomography \cite{renes2004symmetric}.
The Gerzon bound \cite{LemSeiGre1991} states that there exists an equiangular tight frame of $n$ vectors in $\C^d$ only if $n \leq d^2$.
Zauner conjectured in his doctoral thesis\cite{Zau1999} that for every dimension $d$, there exists an equiangular tight frame that saturates the Gerzon bound.
Such an equiangular tight frame is also known as a \textbf{symmetric informationally complete positive operator-valued measure (SIC)}.

In the sequel, we identify $\C^d$ with the space of complex-valued functions over $\mathbb{Z}_d:=\mathbb{Z}/d\mathbb{Z}$.
Put $\omega := e^{2 \pi i / d}$, and define the \textbf{translation} and \textbf{modulation} operators $T,M\colon\C^d \to \C^d$ by
\[
(Tv)(j) = v(j-1),
\qquad
(Mv)(j) = \omega^j\cdot v(j),
\qquad
v\in\mathbb{C}^d.
\]
It is straightforward to verify that $G(v):=\{M^\ell T^k v\}_{k,\ell=0}^{d-1}$ is a tight frame with frame bound $d\|v\|^2$ for every choice of $v\neq 0$.
We refer to $G(v)$ as the \textbf{Gabor frame} generated by $v$. 
When $G(v)$ is equiangular, we say that $v$ is a \textbf{fiducial vector}.
In his conjecture, Zauner actually predicted the existence of a fiducial vector (of a particular form) in $\mathbb{C}^d$ for every $d$.
As a consequence of the theory of projective $t$-designs \cite{roy2007weighted}, it holds that
\begin{equation}
\label{eq.potential}
\frac{1}{d}\sum_{k,\ell=0}^{d-1}|\langle v,M^\ell T^kv\rangle|^4
\geq\frac{2}{d+1},
\end{equation}
with equality precisely when $v$ is a fiducial vector.
As such, one can hunt for fiducial vectors by numerically minimizing the left-hand side of \eqref{eq.potential}, and in fact, this approach has been used to identify putative fiducial vectors to machine precision for every $d\leq 151$, and for a handful of larger dimensions \cite{fuchs2017sic}.
We say ``putative fiducial vectors'' because it is possible (albeit unlikely) that there is no solution to the defining system of polynomials that resides in a neighborhood of the numerical solution; a guarantee to the contrary would require a version of the \L{ojasiewicz} inequality \cite{ji1992global} with explicit constants.

In order to identify honest fiducial vectors, one is inclined to solve the defining system of polynomials, and to this end, solutions have been obtained by Gr\"{o}bner basis calculation \cite{scott2010symmetric}.
However, calculating a Gr\"{o}bner basis for even modest polynomial systems requires a substantial amount of memory and runtime, and so progress with this approach quickly stalled.
Interestingly, the resulting fiducial vectors exhibit some predictable field structure \cite{appleby2017sics}, and these observations have been leveraged to systematically promote numerical solutions to exact solutions \cite{appleby2018constructing} in dimensions that are much too large for Gr\"{o}bner basis calculation.
At this point, a new bottleneck has emerged: the naive description length of fiducial vectors grows quickly with the dimension.
Case in point, the exact coordinates of one fiducial in dimension $d=48$ ``occupies almost a thousand A4 pages (font size 9 and narrow margins)'' \cite{appleby2018constructing}.
Figure~\ref{fig:main} illustrates that the description length appears to scale like $d^4$.
Presumably, these coordinates enjoy a more compact description in some other representation.
For example, the number fields in which the known fiducial coordinates reside are conjectured to be generated by Stark units \cite{stark1976functions}.
Recently, Kopp \cite{kopp2018sic} leveraged Stark units to formulate a conjectural construction of fiducial vectors in prime dimensions $d\equiv 2\bmod 3$, and this construction produced the first known exact solution in dimension $d=23$.

\begin{figure}[t]
\centering
\includegraphics[width=0.95\textwidth]{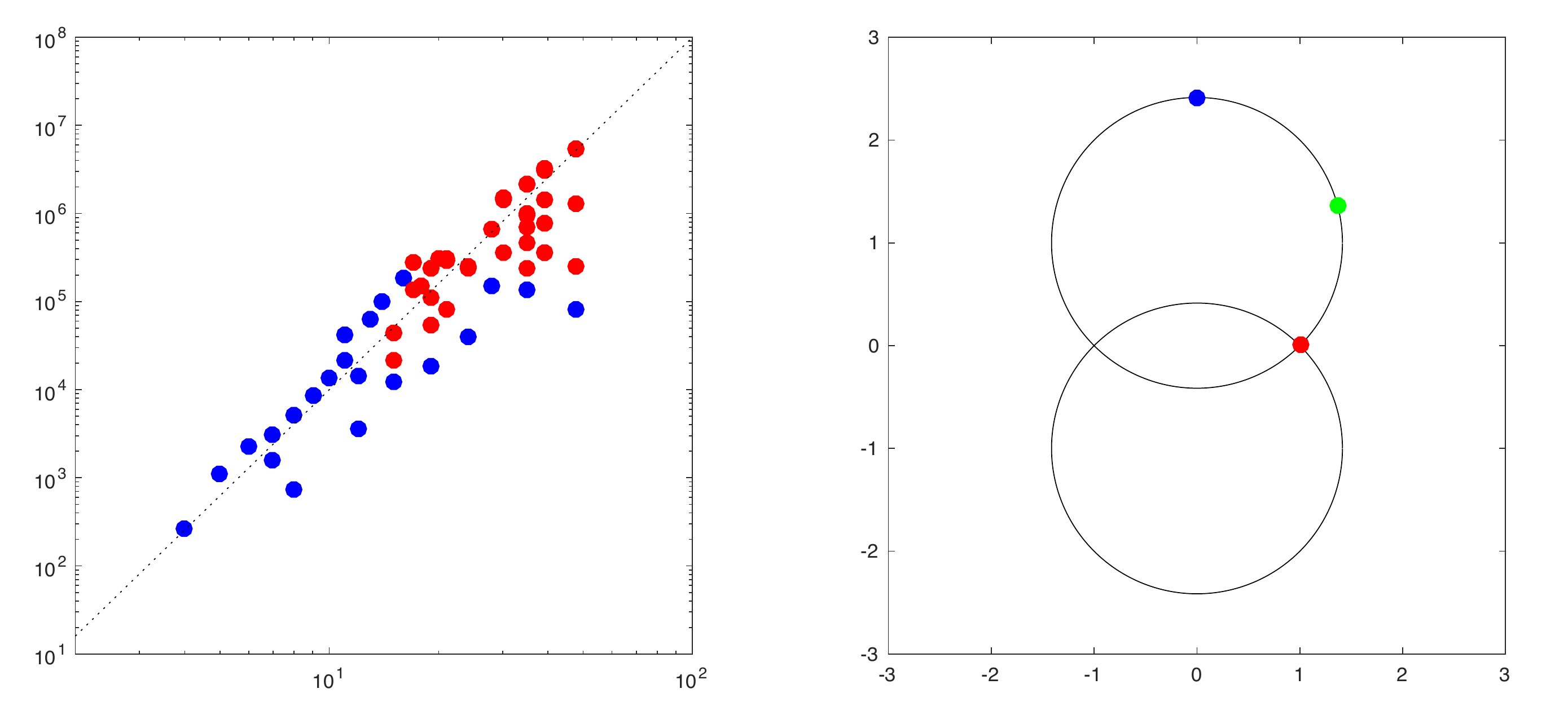}
\caption{\textbf{(left)}
Description lengths of naive expressions of fiducial vectors of SICs.
For each fiducial reported \cite{Zauner:online,Flammia:online}, count the number of characters used to describe the coordinates of the fiducial vector, and plot the results.
Blue dots correspond to solutions obtained by Gr\"{o}bner basis calculation, and red dots correspond to solutions obtained from promoting numerical solutions.
The horizontal axis corresponds to the dimension $d$, while the vertical axis denotes the description length.
The dotted line plots the curve $d^4$.
This illustrates that the naive representation of fiducial coordinates in terms of radicals is inefficient, and so a more compact representation (in terms of Stark units\cite{stark1976functions}, say) is necessary before one can find a constructive proof of Zauner's conjecture.
\textbf{(right)}
Illustration of the proposed method.
The set of $(x,y)$ for which $G((1,x+yi))$ is biangular equals the union of two intersecting circles, plotted in black.
The blue dot at $(x,y)=(0,1+\sqrt{2})$ corresponds to a Gabor MUB, while the red dot at $(x,y)=(1,0)$ corresponds to the trivial biangular Gabor frame $G(\mathbf{1})$.
The angles $\alpha$ and $\beta$ satisfy $\alpha<\beta$ in the former case and $\alpha>\beta$ in the latter case.
As such, for any curve connecting these two points, there exists a point $(x^\star,y^\star)$ at which $\alpha=\beta$, i.e., $G((1,x^\star+y^\star i))$ is a SIC.
Such a point is plotted in green.
}
\label{fig:main}
\end{figure}

Overall, the community appears to be converging towards a constructive proof of Zauner's conjecture that is conditional on the Stark conjectures.
As an unconditional alternative, one might entertain the possibility of a non-constructive proof.
One idea along these lines, posed by Peter Shor on MathOverflow \cite{Shor:online}, is to leverage some sort of geometric fixed point theorem; sadly, no progress in this direction has been made public.
In this paper, we propose another possible route towards a non-constructive proof.
In particular, we relax the set of equiangular Gabor frames to a set of \textit{biangular} Gabor frames.
This larger set includes well-known constructions of \textit{mutually unbiased bases}.
We observe that this set is frequently one-dimensional, which opens up the possibility of a proof of Zauner's conjecture by the intermediate value theorem.

\section{The proposed approach}
\label{sec.approach}

In this section, we outline an approach to prove Zauner's conjecture using the intermediate value theorem.
We say $G(v)$ is \textbf{biangular} if there exists $\alpha$ and $\beta$ such that
\begin{itemize}
\item[(i)]
$|\langle v, T^k v \rangle|^2 = \alpha$ for every $k\in\{1,\ldots,d-1\}$, and
\item[(ii)]
$|\langle v, M^\ell T^k v \rangle|^2 = \beta$ for every $k\in\{0,\ldots,d-1\}$ and $\ell\in\{1,\ldots,d-1\}$.
\end{itemize}
In this case, we can be more precise by saying that $G(v)$ is $(\alpha,\beta)$-biangular.
We note that the angle parameters $\alpha$ and $\beta$ depend on one another:

\begin{lem}
\label{lem.angle balance}
If $G(v)$ is an $(\alpha,\beta)$-biangular Gabor frame for $\mathbb{C}^d$, then $\alpha+d\beta=\|v\|_2^4$.
\end{lem}

\begin{proof}
By tightness, we have
\[
d\|v\|_2^4
=\sum_{k,\ell=0}^{d-1}|\langle v,M^\ell T^k v\rangle|^2
=\|v\|_2^4+(d-1)\alpha+(d^2-d)\beta,
\]
and so rearranging gives the result.
\end{proof}

It is helpful to consider a few examples of biangular Gabor frames:

\begin{example}
\begin{itemize}
\item[(a)]
Let $\mathbf{1}$ denote the all-ones vector in $\mathbb{C}^d$.
Then $G(\mathbf{1})$ is biangular.
\item[(b)]
Suppose $G(v)$ is a $(0,1/d)$-biangular Gabor frame in $\mathbb{C}^d$.
Then each $\{M^\ell T^kv\}_{k=0}^{d-1}$ is an orthonormal basis, and together, these bases are \textit{mutually unbiased} \cite{planat2006survey}.
For example, if $d\geq 5$ is prime and $v$ is the Fourier transform of the corresponding \textit{Alltop sequence} \cite{alltop1980complex}
\[
f(t):=\frac{1}{\sqrt{d}}\cdot e^{2\pi it^3/d},
\qquad
t\in\{0,\ldots,d-1\},
\]
then $G(v)$ is $(0,1/d)$-biangular.
We refer to such $G(v)$ as \textbf{Gabor MUBs}.
\item[(c)]
If $G(v)$ is equiangular, then $G(v)$ is biangular with $\alpha=\beta$.
\item[(d)]
If $G(v)$ is biangular, then $G(cv)$ is also biangular for every $c\in\mathbb{C}^\times$.
\end{itemize}
\end{example}

Let $B_d$ denote the real algebraic variety of $v\in\C^d$ for which $G(v)$ is biangular.
Perhaps surprisingly, we observe that $B_d/\mathbb{C}^\times$ is at times one-dimensional even though $B_d$ is defined by $\Omega(d^2)$ polynomials over $2d$ real variables.
We suspect that this feature can be leveraged to prove the existence of SICs.
For example, the following result allows us to promote MUBs to SICs:

\begin{lem}
\label{lem.mub+connected}
Suppose there exists a Gabor MUB in $\mathbb{C}^d$ and $B_d$ is path-connected.
Then there exists a SIC in $\mathbb{C}^d$.
\end{lem}

\begin{proof}
Select $v_0$ such that $G(v_0)$ is an MUB, put $v_1=\frac{1}{\sqrt{d}}\mathbf{1}$.
Then by path-connectivity, there exists a parameterized curve $v\colon[0,1]\to\mathbb{C}^d$ such that $v(0)=v_0$, $v(1)=v_1$, and $G(v(t))$ is biangular for every $t\in(0,1)$.
Without loss of generality, it holds that $\|v(t)\|_2=1$ for every $t$.
Define $\alpha,\beta\colon[0,1]\to\mathbb{R}$ such that
\[
\alpha(t)
:=|\langle v(t), T v(t) \rangle|^2,
\qquad
\beta(t)
:=|\langle v(t), MT v(t) \rangle|^2,
\qquad
t\in[0,1].
\]
By Lemma~\ref{lem.angle balance}, it holds that
\[
\Delta(t)
:=\beta(t)-\alpha(t)
=\frac{1-\alpha(t)}{d}-\alpha(t)
=\frac{1-(d+1)\alpha(t)}{d}.
\]
Considering $\alpha(0)=0$ and $\alpha(1)=1$, then the continuous function $\Delta\colon[0,1]\to\mathbb{R}$ satisfies $\Delta(0)=1/d>0$ and $\Delta(1)=-1<0$.
The intermediate value theorem then guarantees the existence of $t^\star\in(0,1)$ such that $\Delta(t^\star)=0$, i.e., $\alpha(t^\star)=\beta(t^\star)$.
As such, $G(v(t^\star))$ is equiangular, i.e., the claimed SIC.
\end{proof}

Importantly, Gabor MUBs (unlike SICs) are known to exist in infinitely many dimensions.
The bottleneck of applying Lemma~\ref{lem.mub+connected} is demonstrating path-connectivity.
The following provides a sufficient condition to this end:

\begin{lem}
\label{lem.sufficient connectivity}
If $C_d:=\{v\in B_d:v(0)=1\}$ is path-connected, then $B_d$ is path-connected.
\end{lem}

\begin{proof}
Suppose $C_d$ is path-connected, and for each $j\in\mathbb{Z}_d$, denote $S_j:=\{v\in B_d:v(j)=1\}$ so that $S_0=C_d$.
Then by symmetry, every $S_j$ is path-connected.
To see that $B_d$ is also path-connected, pick any $v_0,v_1\in B_d$.
For each $k\in\{0,1\}$, we have that $v_k$ is nonzero by assumption, and so one of its coordinates is nonzero, say, coordinate $j_k\in\mathbb{Z}_d$.
Let $c_k\colon[0,1]\to\mathbb{C}^\times$ denote any parameterized curve in $\mathbb{C}^\times$ from $c_k(0)=1$ to $c_k(1)=v_k(j_k)$.
Then $v_k(t):=v_k/c_k(t)$ is a curve in $B_d$ such that $v_k(0)=v_k$ and $v_k(1)\in S_{j_k}$.
As such, we can traverse from $v_0$ to $v_0/v_0(j_0)$ along $v_0(\cdot)$, and then from $v_0/v_0(j_0)$ to $\mathbf{1}$ by the path-connectivity of $S_{j_0}$, and then from $\mathbf{1}$ to $v_1/v_1(j_1)$ by the path-connectivity of $S_{j_1}$, and then from $v_1/v_1(j_1)$ to $v_1$ along the reversal of $v_1(\cdot)$.
\end{proof}

As a proof of concept, we leverage the above results to prove the (well-known) existence of a SIC in $\mathbb{C}^2$.
(Importantly, our proof is non-constructive, unlike the usual proof.)

\begin{cor}
\label{cor.C2}
There exists a SIC in $\mathbb{C}^2$.
\end{cor}

\begin{proof}
Put $u=(1,(1+\sqrt{2})i)$.
It is straightforward to verify that $G(u/\|u\|_2)$ is an MUB.
We will demonstrate that $C_2$ is path-connected so that the result follows from Lemmas~\ref{lem.mub+connected} and~\ref{lem.sufficient connectivity}.
To this end, note that $v\in C_2$ if and only if there exist $x,y,\alpha,\beta\in\mathbb{R}$ such that $v=(1,x+yi)$ and
\[
4x^2=\alpha,
\qquad
4y^2=\beta,
\qquad
(1-x^2-y^2)^2=\beta.
\]
In other words, $C_2$ is the set of all $(1,x+yi)$ such that $x^2+(y\pm1)^2=2$.
Geometrically, this is the union of two circles of radius $\sqrt{2}$ centered at $(0,\pm1)$; see Figure~\ref{fig:main} for an illustration.
Since these circles intersect, it follows that $C_2$ is path-connected, as desired.
\end{proof}

To prove Zauner's conjecture, we would need to replicate this non-constructive proof technique in every dimension.
This suggests the following:

\begin{problem}
\label{prob.connectivity}
For which dimensions $d$ is $B_d$ is path-connected?
\end{problem}

There has already been some work to prove path-connectivity of certain varieties of frames.
Most work along these lines has focused on the variety of unit norm tight frames.
Initial work \cite{cahill2017connectivity} leveraged so-called \textit{eigensteps} \cite{cahill2013constructing} to construct explicit paths that demonstrate path-connectivity, whereas a more recent treatment \cite{needham2018symplectic} exploits technology from symplectic geometry to obtain a non-constructive proof.
It would be interesting if similar technology could be applied to tackle Problem~\ref{prob.connectivity}.

Next, MUBs are only known to exist in prime power dimensions, and so we would need to improve Lemma~\ref{lem.mub+connected} before we can hope to prove Zauner's conjecture.
In fact, ``Gabor MUB'' in Lemma~\ref{lem.mub+connected} can be replaced by any biangular Gabor frame with appropriately small $\alpha$, suggesting the following:

\begin{problem}
\label{prob.construction}
For every $d$, find $v\in\mathbb{C}^d$ such that both $\|v\|_2=1$ and $G(v)$ is $(\alpha,\beta)$-biangular with $\alpha<\frac{1}{d+1}$.
\end{problem}

Considering the successful instance of Gabor MUBs, it seems reasonable to suspect that Problem~\ref{prob.construction} can be solved in closed form, even though the $\alpha=\frac{1}{d+1}$ case of SICs has resisted such a solution.
Finally, note that we do not require all of $B_d$ to be path-connected, as it suffices to find $v_0$ and $v_1$ for which there exist $\alpha_0$ and $\alpha_1$ such that
\begin{itemize}
\item[(i)]
$G(v_j)$ is $(\alpha_j,\frac{1}{d}(1-\alpha_j))$-biangular for each $j\in\{0,1\}$,
\item[(ii)]
$v_0$ and $v_1$ are path-connected in $B_d$, and
\item[(iii)]
$\alpha_0 < \frac{1}{d+1} < \alpha_1$.
\end{itemize}
Of course, it is likely easier to solve Problem~\ref{prob.connectivity}.

To illustrate our observation that biangular Gabor frames enjoy path-connectivity, we run a simple numerical experiment:
For each $d\in\{2,4,5\}$, we consider the numerical fiducial reported by Scott and Grassl \cite{scott2010symmetric} (when $d=3$, the variety of SIC fiducials is already interesting).
Call this vector $v_0$.
We slightly perturb this fiducial and then locally minimize the sum of the squares of the polynomials that define the variety of biangular Gabor seed vectors.
This produces a new point $v_1$ on the variety.
Next, we locally minimize from the perturbation $v_j+c\cdot\frac{v_j-v_{j-1}}{\|v_j-v_{j-1}\|_2}$ to obtain $v_{j+1}$ (with $j=1$), and we iterate this procedure to identify a sequence of points on the variety.
(Here, $c$ is a small constant.)
The results of this experiment are illustrated in Figure~\ref{fig:numerics}.

\begin{figure}[t]
\centering
\includegraphics[width=0.95\textwidth]{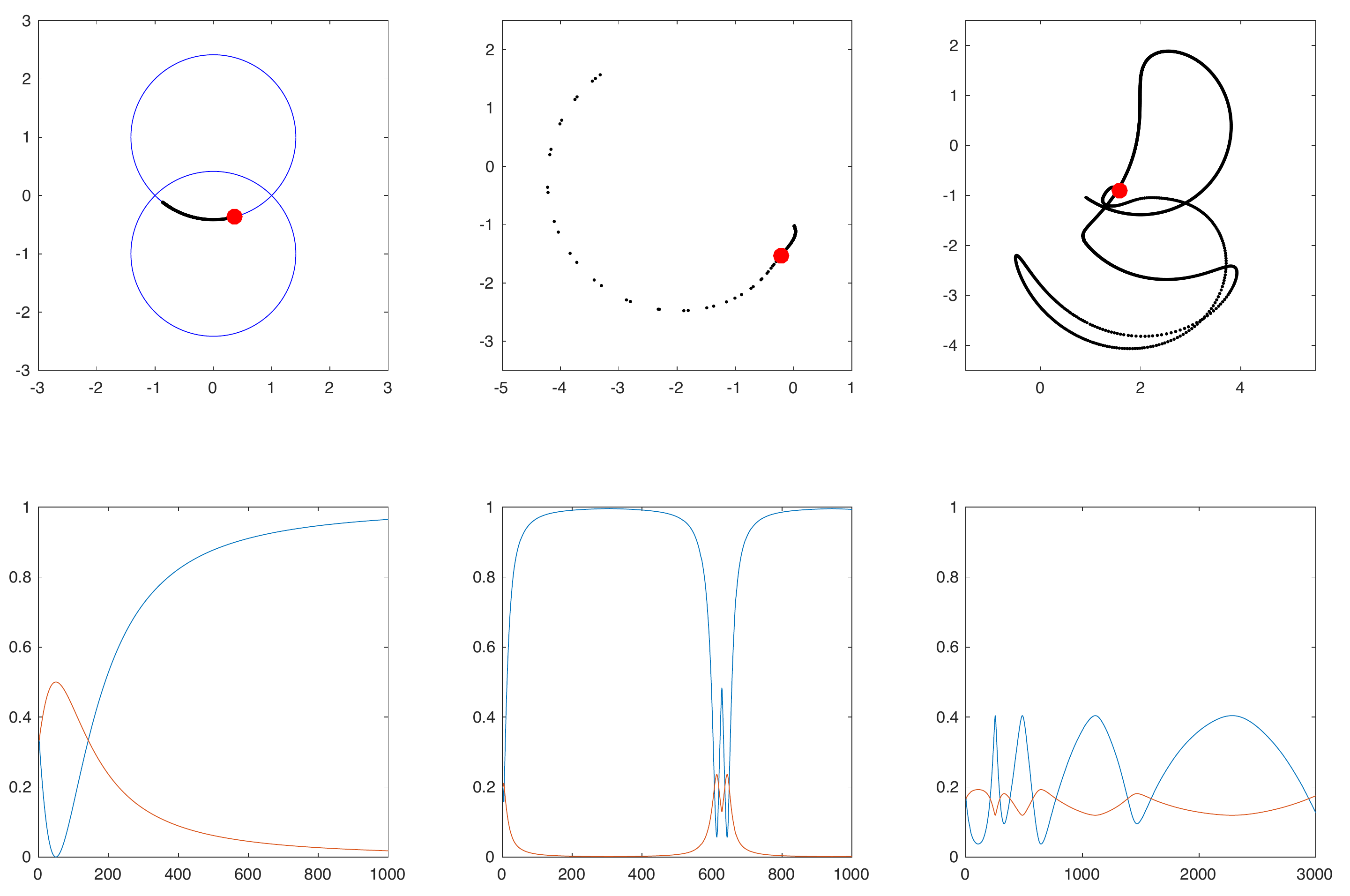}
\vspace{30pt}
\caption{Numerical experiments to illustrate path-connectivity in the variety of biangular Gabor frames.
\textbf{(left)}
As a control, we first consider the case of $d=2$, which we have already treated in the proof of Corollary~\ref{cor.C2}.
In blue, we plot the set of all $(x,y)$ for which $(1,x+yi)$ generates a biangular Gabor frame.
The red dot in this set corresponds to the numerical fiducial reported by Scott and Grassl \cite{scott2010symmetric}.
We then traverse the variety using the numerical scheme discussed at the end of Section~\ref{sec.approach}; we plot the corresponding trajectory in black.
In the display below, we plot how $(\alpha,\beta)$ evolve over this trajectory (we compute these angles after normalizing $v$ to have unit norm).
The angles start at $(\frac{1}{3},\frac{1}{3})$, corresponding to the SIC, then pass through $(0,\frac{1}{2})$, corresponding to an MUB, and then finally approach $(1,0)$, corresponding to a trivial Gabor frame.
We repeat this experiment for $d=4$ \textbf{(middle)} and $d=5$ \textbf{(right)}, fixing $v(0)=1$, plotting the trajectory of $(\operatorname{Re} v(1),\operatorname{Im} v(1))$ above, and then plotting the angles $(\alpha,\beta)$ below.
Unlike the $d=2$ case, we do not have analytic expressions for the variety in these cases.
}
\label{fig:numerics}
\end{figure}

\section{Discussion}

This paper proposed a new approach to tackle Zauner's conjecture.
Specifically, we relax the set of SICs to a larger set of biangular Gabor frames, which appear to form a path-connected variety.
This feature could very well allow for a non-constructive proof of Zuaner's conjecture, and we isolate Problems~\ref{prob.connectivity} and~\ref{prob.construction} as steps towards this end.
In addition, it would also be interesting to leverage the variety of biangular Gabor frames to facilitate the search for numerical SICs.
We leave these investigations for future work.

\section*{ACKNOWLEDGMENTS}
The authors thank John Jasper and Hans Parshall for commenting on a draft of this paper.
MM and DGM were partially supported by AFOSR FA9550-18-1-0107. DGM was also supported by NSF DMS 1829955 and the Simons Institute of the Theory of Computing.

\bibliography{biangle} 
\bibliographystyle{spiebib} 
\end{document}